\documentclass[12pt,oneside]{amsart}
\usepackage[margin=1in]{geometry}

\usepackage{graphicx}
\usepackage{amssymb}
\usepackage{tikz}
\usepackage{tikz-cd}
\usepackage{amsmath,amsthm,amscd,amssymb}
\usepackage{latexsym}
\usepackage[colorlinks,citecolor=red,pagebackref,hypertexnames=false]{hyperref}
\hypersetup{%
  colorlinks = true,
  citecolor=red,
  linkcolor  = black
}

\numberwithin{equation}{section}
\theoremstyle{plain}
\newtheorem{theorem}{Theorem}[section]
\newtheorem*{theorem*}{Theorem}
\newtheorem{lemma}[theorem]{Lemma}
\newtheorem{corollary}[theorem]{Corollary}
\newtheorem{proposition}[theorem]{Proposition}

\theoremstyle{definition}
\newtheorem{definition}[theorem]{Definition}

\theoremstyle{remark}
\newtheorem{remark}[theorem]{Remark}

\newtheorem{case[theorem]}{Case}

\newcommand{\Slim}[1]{\sum\limits_{#1}}
\newcommand{\Norm}[1]{\left\lvert\left\lvert {#1} \right\rvert\right\rvert}
\newcommand{\Abs}[1]{\left\lvert{#1}\right\rvert}

\newcommand{\Field}[1]{\mathbb{F}_{#1}}

 \usepackage{hyperref}

\title[Square-Root Cancellation, Dot Products, and Finite Rings]{Square-Root Cancellation, Averages over Hyperplanes, and the Structure of Finite Rings}
\author{Nathaniel Kingsbury-Neuschotz}
\date{\today}
\begin{document}
\begin{abstract}
We formulate a notion of square-root cancellation for the operator which sums a mean-zero function over a rotating hyperplane in $R^d$, where $R$ is a  possibly noncommutative finite ring. Using an argument due to Hart, Iosevich, Koh, and Rudnev, we show that this square-root cancellation occurs uniformly when $R$ is a finite field. We then show that this square-root cancellation cannot occur uniformly over families of finite rings which are not eventually finite fields. This extends an earlier result of the author to a non-translation-invariant operator.

\end{abstract}
\maketitle
\section{Introduction}
The square-root law is a useful and ubiquitous heuristic which predicts the following type of bound:
$$\left|\sum\text{oscillating terms of modulus 1}\right|\leq C\sqrt{\text{\# of terms}}.$$
This heuristic arises all over analysis. For example, with high probability a one-dimensional discrete random walk will be within distance $C\sqrt{N}$ of the origin after $N$ steps. If one takes a random subset $S$ of $(\mathbb{Z}/N\mathbb{Z})^d$, then its Fourier transform will satisfy a square-root bound up to a logarithmic factor (\cite{Babai} theorem 5.14). For a number-theoretic manifestation, the Riemann hypothesis is equivalent to the statement that 
$$\Slim{n\leq N}\mu(n)=O(N^{\frac{1}{2}+\epsilon})$$
for all $\epsilon$ greater than $0$, where $\mu$ is the M\"obius function; for a more intuitive interpretation, in this equivalence the M\"obius function $\mu$ may be replaced by Liouville's function $\lambda$, so that the Riemann hypothesis roughly states that the parity of the number of prime factors of the integers behaves like a sequence of independent unbiased Bernoulli trials. The square-root law has many further manifestations in number theory; the interested reader should see Mazur's wonderful survey \cite{Mazur}.

In this paper, we will work with an operator-theoretic formulation of the square-root law. Suppose that $A$ is an operator which sums a \textit{mean-zero} function $f(x)$ over a set $S = S(x)$ of dimension $d_S$ not dependent on $x$ within the vector space $\Field{q}^d$. The square-root law predicts that ``most'' such operators should obey a bound of the form 
$$||Af||\leq Cq^{d_S/2}||f||.$$
For example, if $S(x) = \{y|x-y\in S_0\} = x-S_0$ for a fixed set $S_0$ of dimension $d_S$ satisfying the Fourier decay condition\footnote{See the opening of section 2 for our formulation of the Fourier transform} $\widehat{S_0}(\chi)\leq Cq^{-d}|S_0|^\frac{1}{2}$ whenever $\chi$ is a nontrivial character, then convolution, Parseval, mean-zero-ness of $f$, and the Fourier decay condition give that $||Af||_2 = ||S_0*f||_2 = \leq Cq^\frac{d_S}{2}||f||_2.$ According to standard terminology, when $S_0$ satisfies the bound $\widehat{S_0}(\chi)\leq Cq^{-d}|S_0|^\frac{1}{2}$ for $\chi$ a nontrivial character, $S_0$ is said to be $C$-Salem.\footnote{This condition demands that the indicator function of the set have essentially optimal Fourier decay. The term ``Salem Set'' to refer to sets with essentially optimal Fourier decay was first used in the continuous setting, where Hausdorff dimension constrains Fourier decay; see \cite{Bluhm} and \cite{Mattila} for a detailed definition in that context. In \cite{IoRud}, Iosevich and Rudnev defined Salem sets in the discrete setting as those whose Fourier transforms satisfy the bound $\widehat{S_0}(\chi)\leq Cq^{-d}|S_0|^\frac{1}{2}$ for all nontrivial characters $\chi,$ which is the definition used above. This definition has become standard in the discrete setting (as in \cite{Chen}, \cite{Fraser}, \cite{IoMorgPak}, \cite{IoMPak}, and \cite{KohShen}); such uniform bounds are extremely useful for arguments in geometric combinatorics, as in  \cite{IoMorgPak}, \cite{IoRud}, and \cite{KohShen}, and somewhat subtly in our proof of theorem \ref{EdotE}. Despite the strength of the optimal Fourier decay condition, Salem sets are common: a random subset of a finite field obeys (up to a logarithmic factor) a Salem-type bound (\cite{Babai} theorem 5.14), and the graph of a generic polynomial function over a finite field of degree at least 2 is a true Salem set (\cite{Kings}, subsection 2.2, and especially theorem 2.7 and corollaries 2.8 and 2.9).} By analogy, we make the following definition, which we phrase over an arbitrary finite ring $R$ for convenience.
\begin{definition}\label{formalSalemDef}
    Let $R$ be a finite ring, and let $A$ be the operator such that
    $$Af(x) = \sum\limits_{y\in S(x)}f(y),$$
    where $S(x)$ is a family of $d_S$-dimensional subsets of $R^d$.\footnote{It is not critical to be precise about what notion of dimension we mean here, as shortly we will focus on a very specific family of operators; that said, for the purposes of this paper we may define dimension by requiring that $S(x)$ be defined by a system of polynomial equations in $x, y$ with coefficients in $\mathbb{Z}$, thus giving a scheme $X$ with a projection down to $\mathbb{A}^d(\mathbb{Z})$ by projection onto the variable $x$, and say that the dimension $d_S$ is the relative dimension of this morphism. Ultimately this notion of dimension is not defined with reference to the ambient ring $R$, so that it avoids any noncommutative ring theory even when $R$ is noncommutative.} We say that $A$ is $C$-Salem if
    $$||Af||_2\leq C|R|^{d_S/2}||f||_2,$$
for all mean-zero functions $f$.
\end{definition}

We should remark briefly on the sorts of operators we might expect to be Salem. We will specifically be interested in the setting where $S(x) = \{y: \Phi(x, y) = 0\}$ for some polynomial function $\Phi(x, y)$. In this setting, we should think of $S(x)$ as $d-1$ dimensional, so that the averaging operator $A$ is $C$-Salem if
\begin{equation}\label{SalemDef}
    ||Af||_2 \leq C|R|^\frac{d-1}{2}||f||_2
\end{equation}
for all mean-zero functions $f$ on $R^d$. For comparison, the following ``trivial'' bound holds for all functions:
\begin{proposition}
    Fix a polynomial function $\Phi(x, y)$. Assume that for all $x$, the set 
    $$S(x) = \{y:\Phi(x, y) = 0\}$$
    satisfies
    $$|S(x)| \leq c|R|^{d-1}$$
    and the set
    $$S'(x) = \{y:\Phi(y, x) = 0\}$$
    satisfies 
    $$|S'(x)| \leq c|R|^{d-1}.$$
    Then for any function $f: R^d\rightarrow \mathbb{C}$, we have that 
    $$||Af||_2 \leq c|R|^{d-1}||f||_2$$
\end{proposition}
\begin{proof}
Recall that $A$ is the operator

$$Af(x) = \sum\limits_{y\in S(x)}f(y).$$

Let $A'$ be the operator

$$A'f(x) = \sum\limits_{y\in S'(x)}f(y).$$

Let $||A||_\infty$ denote the operator norm of $A$ induced from the $L^\infty$ norm on the space of functions $R^d\rightarrow \mathbb{C}$, and let $||A'||_\infty$ be defined similarly. We claim $||A||_\infty \leq c|R|^{d-1}$ and $||A'||_\infty \leq c|R|^{d-1}$. Indeed, for each $x$, we have that 

\[\begin{split}
|Af(x)| &= \Abs{\sum\limits_{y\in S(x)}f(y)}\\
&\leq \text{sup}_y|f(y)|\#S(x)\\
&\leq c||f||_\infty|R|^{d-1};\end{split}\]

the bound for $||A'||_\infty$ may be proven similarly. It follows that

$$||A'A||_\infty\leq c^2|R|^{2d-2},$$

so that the largest eigenvalue of $A'A$ is at most $c^2|R|^{2d-2}$. But the matrix of $A'$ with respect to the natural coordinates is the (conjugate) transpose of that of $A$, so that the largest singular value of $A$ is at most $c|R|^{d-1}$; it follows that the $L^2$ norm of $A$ is also at most $c|R|^{d-1}$.
\end{proof}

In comparison, the $C$-Salem bound (\ref{SalemDef}) on the action of $A$ on mean-zero functions is a $\frac{d-1}{2}$ power savings; by the correspondence between Fourier decay and smoothness, this is analogous to the gain of $\frac{d-1}{2}$ derivatives in the \textit{Sobolev estimate}
$$||Tf||_{L_\frac{d-1}{2}^2(\mathbb{R}^d)}\leq C||f||_{L^2(\mathbb{R}^d)}$$
where $Tf(x) = \int_{\Gamma_x}f(y)d\sigma_x(y)$ averages $f$ over a continuous family of smooth manifolds of dimension $d-1$ in $\mathbb{R}^d$. In the continuous setting, the central result of \cite{PhongStein} guarantees the above Sobolev estimate will hold when $\Gamma_x = \{y:\Phi(x, y) = 0\}$ for a smooth function $\Phi$ satisfying the Phong-Stein rotational curvature condition, namely, that the Monge-Amp\`ere determinant
\[
\mathcal{M}(\Phi) =
 \begin{vmatrix}
0 & \frac{\partial\Phi}{\partial x_1} & \dots &  \frac{\partial\Phi}{\partial x_d}\\
 \frac{\partial\Phi}{\partial y_1} &  \frac{\partial^2\Phi}{\partial x_1\partial y_1} & \dots & \frac{\partial^2\Phi}{\partial x_d\partial y_1}\\
\vdots & \vdots & \ddots & \vdots\\
 \frac{\partial\Phi}{\partial y_d} & \frac{\partial^2\Phi}{\partial x_1\partial y_d} & \dots & \frac{\partial^2\Phi}{\partial x_d\partial y_d}
\end{vmatrix}
\]
restricted to the set where $\Phi(x, y) = 0$ does not vanish. In our setting, a more precise heuristic might be that when the operator $A$ sums the function $f$ over the set $S(x) = \{y:\Phi(x, y) = 0\}$ where $\Phi(x,y)$ is a polynomial function over $\Field{q}$ satisfying the rotational curvature condition of Phong and Stein, the corresponding averaging operator $A$ is likely to satisfy the $C$-Salem bound (\ref{SalemDef}).\footnote{For an instance of the results of Phong and Stein applied in geometric combinatorics in the continuous setting, see \cite{IoJoLaba}.}

In this paper, we will be specifically concerned with the operator $A_t$ defined by 
$$A_tf(x) = \Slim{y:y\cdot x = t}f(y),$$ where $t$ is a unit of $\Field{q}$ or a more general finite ring. This operator averages (really, sums) functions $f$ over hyperplanes which rotate as $x$ varies; it arises naturally in the study of the Erd\H{o}s-Falconer Distance Conjecture and the sum-product phenomenon in vector spaces over finite fields, see for instance \cite{HartIoKohRud} and the references contained therein. In the setting of finite fields, techniques from \cite{HartIoKohRud} establish that $A_t$ is $\sqrt{2}$-Salem; see our theorem \ref{finiteFieldsBound} and the discussion beforehand for details. This operator fits into the above paradigm with $\Phi(x, y) = y\cdot x - t$; we remark that the Monge-Amp\`ere determinant is 
\[
\begin{vmatrix}
0 & y_1 & y_2 & \cdots & y_d\\
x_1 & 1 & 0 & \cdots & 0\\
x_2 & 0 & 1 & \cdots & 0\\
\vdots & \vdots & \vdots & \ddots & \vdots\\
x_d & 0 & 0 & \cdots & 1
\end{vmatrix} = -x_1y_1 - x_2y_2 -\dots -x_dy_d;
\]
this equals $-t$ when restricted to the set where $x\cdot y = t$, and $t$ is a unit so $-t\neq 0$. 

In previous work (\cite{Kings}), the author showed that square-root cancellation for exponential sums over certain deterministic sets which is known to occur over finite fields does not occur over other finite rings, building on the foundational work of Iosevich, Murphy, and Pakianathan in \cite{IoMPak}.\footnote{These works build on the paradigm of studying arithmetic and geometric combinatorics over general rings, rather than just $\mathbb{Z}$, $\mathbb{R},$ and $\mathbb{F}_q$; for another example see Tao's work on the sum-product phenomenon in general rings \cite{Tao}.} It follows from these results that given a fixed constant $C$, the averaging operator $A$ corresponding to $S(x) = x - S_0$ where $S_0$ is a graph of a polynomial equation does not satisfy the $C$-Salem bound (\ref{SalemDef}) for sufficiently large finite rings which are not finite fields or matrix rings of small dimension\footnote{In the case where $S_0$ is the \textit{paraboloid} $x_d = x_1^2 + \dots + x_{d-1}^2 + c$, the operator $A$ is only $C$-Salem for sufficiently large finite rings which are specifically finite fields, i.e., matrix rings of small dimension do not enter the picture.}. In this paper, we establish a similar result for the operator $A_t$. In particular, we show the following:
\begin{theorem}
    Fix an integer $d\geq 2$. Suppose that $\{(R_i, t_i)\}_{i = 1}^\infty$ is a sequence of pairs of rings $R_i$ of size tending to infinity and units $t_i\in R_{i}^\times$, and that $A_{t_i}$ satisfies the $C$-Salem bound (\ref{SalemDef}) for some $C$ independent of $i$. Then the rings $R_i$ must eventually be fields.
\end{theorem}
Equivalently (and as formulated in theorem \ref{GeneralCaseDots}), for each fixed $C>0$, there are only finitely many pairs $(R, t)$ with $R$ a ring and $t\in R^\times$ where $R$ is not a finite field and $A_t$ satisfies the $C$-Salem bound (\ref{SalemDef}). This provides an extension of the failure of square-root cancellation in the presence of zero-divisors beyond the context of convolution operators.
\subsection{\texorpdfstring{Dependence on $t$}{Dependence on t}}
\textit{A priori}, the properties of $A_t$ may depend on the choice of $t$. It turns out that for our purposes, the behavior is the same so long as $t$ is a unit. 
\begin{proposition}\label{allUnitsCreatedEqual}
    Let $R$ be a finite ring, and let $t$ and $t'$ be units. Then
    $$||A_t||_2 = ||A_{t'}||_2$$
\end{proposition}
\begin{proof}
    Let $f$ be a mean-zero function on $R^d$. Let $g$ be defined by $g(x) = f(tt'^{-1}x).$ As $x\mapsto tt'^{-1}x$ is bijective, we have $||g||_2 = ||f||_2$; we also have
    \[
    \begin{split}
        ||A_tf||_2 &= \left(\Slim{x}\Abs{\Slim{y:y\cdot x = t}f(y)}^2\right)^{1/2} \\
        &= \left(\Slim{x}\Abs{\Slim{y:y\cdot x = t'}f(tt'^{-1}y)}^2\right)^{1/2}\\
        &= \left(\Slim{x}\Abs{\Slim{y:y\cdot x = t'}g(y)}^2\right)^{1/2}\\
        &=\Norm{A_{t'}g}_2\\
        &\leq ||A_{t'}||_2||f||_2.
    \end{split}
    \]
    It follows that $||A_t||_2\leq ||A_{t'}||_2$; interchanging the roles of $t$ and $t'$ gives the result.
\end{proof}
\subsection{Solvability of Equations}\label{solvabilitySection}
The heart of the argument in \cite{Kings} is a statement about the solvability of the equation 
$$x_d-y_d = f(x_1-y_1, x_2-y_2,\dots, x_{d-1}-y_{d-1})$$
subject to a constraint of the form ${x}, {y}\in E$---if the underlying operator $A$ satisfies the $C$-Salem bound (\ref{SalemDef}), one can show that if $E$ contains at least $C|R|^\frac{d+1}{2}$ elements this constrained equation is guaranteed to have a solution. In this work, a similar role is played by a statement about the solvability of 
$$y\cdot x = y_1x_1 + y_2x_2 + \cdots + y_dx_d = t$$
subject to ${x},{y}\in E,$ or in other words, to a statement about how large $E$ must be to guarantee that $$E\cdot E\supseteq R^\times,$$
see proposition \ref{idealBound} for details. We outline these techniques in subsection \ref{Solvability}, as they are quite readily generalized. Unfortunately, arguments based on these sorts of results seem not to be able to rule out small matrix rings, as we briefly discuss in subsection \ref{Solvability}; our proof of theorem \ref{matrixRings} at the start of section \ref{Matrices} avoids this issue through a linear algebraic proof not sensitive to the dimension of the matrix ring. 
\subsection{A Graph-Theoretic Formulation}
In \cite{IoMPak}, Iosevich, Murphy, and Pakianathan defined a family of graphs $G_H(R)$ which they named ``Hyperbola Graphs,'' with vertex set $R^2$ and with two vertices $x$ and $y$ adjacent if their difference lies in the hyperbola
$$H(R) = \{(x_1, x_2)\in R^2|x_1x_2 = 1\}.$$
In our notation, the action of the adjacency matrix of $G_H(R)$ is the averaging operator defined by $S(x) = x-H$. Iosevich, Murphy, and Pakianathan's study of Kloosterman sums over finite rings allows them to produce lower bounds on the spectral gap of the hyperbola graphs, with important dependence on the structure of the underlying ring.

Over a \textit{commutative} ring $R$, our work corresponds to a study of the adjacency matrix of the ``dot-product graphs'' $G_{dot}(R, d, t)$ whose vertex set is $R^d$ and two vertices $x$ and $y$ are connected by an edge if $y\cdot x = t$. Over a \textit{noncommutative} ring, we may have $x\cdot y \neq y\cdot x$, so that this construction only gives a directed graph. Our results give partial information about the spectrum of the adjacency matrices of these (di)graphs, showing that this spectrum has a significant dependency on the structure of the underlying ring $R$.

In the special case where we work over a finite field $\mathbb{F}_q$, the dot product graphs consist of the disjoint union of a singleton (corresponding to the origin in $\mathbb{F}_q^d$) and a $q^{d-1}$-regular graph, as for any nonzero vector $x$ in $\mathbb{F}_q^d$ there are exactly $q^{d-1}$ vectors $y$ such that $y\cdot x = t$, by linear algebra. Because of this regularity, we can easily relate the spectrum of the adjacency matrix of this graph to the spectrum of its Laplacian; in particular, the Laplacian $L$ of this graph kills the subspace of functions constant on each of these two components of the graph; the rest of the spectrum of $L$ comes from the action of $L$ on the space of mean-zero functions $f$ such that $f(0) = 0$. In theorem \ref{finiteFieldsBound} below, we establish that over a finite field, 
$$||A_t||_2 \leq \sqrt{2}q^{(d-1)/2};$$
it immediately follows that for any eigenvalue $\lambda$ of $L$ other than the two known zero eigenvalues we have
$$\lambda \geq q^{d-1}-\sqrt{2}q^\frac{d-1}{2}.$$
In particular, if $d > 2$ or $q > 2$, the $q^{d-1}$-regular component of the dot product graph is connected.
\section{Notation and Preliminaries}
Throughout this paper, $R$ will be a finite ring (unital, with $1\neq 0$, not necessarily commutative).  The \textit{Fourier transform} on $R^d$ is defined as follows: given a function $f:R^d\rightarrow\mathbb{C}$, we define 

$$\widehat{f}(\chi) = \frac{1}{|R|^d}\Slim{x\in R^d}f(x)\overline{\chi(x)}$$
where $\chi$ is a homomorphism from the underlying additive group of $R^d$ to $\mathbb{C}$ (such a function is often called a character of $R^d$). With this definition, we have the Fourier inversion formula

$$f(x) = \Slim{\chi\in (R^d)^{\vee}}\widehat{f}(\chi)\chi(x)$$
where $(R^d)^\vee$ denotes the collection of characters of $R^d$, known as the \textit{dual} of $R^d$.

We will let $V$ denote the vector space of complex-valued functions on $R^d$. By Fourier inversion, $(R^d)^\vee$ forms a basis for $V$. Let $W\subset V$ denote the vector space of \textit{mean-zero functions} on $R^d$; that is, the space of all functions $f$ such that $$\sum_{x\in R^d}f(x) = 0.$$

Throughout this paper, it will be convenient to have separate notations for the operator $A_t$ as it acts on the space of mean-zero functions $W$ and as it acts on the space of all functions $V$. This will allow us to describe $C$-Salemness of $A_t$ (see definition \ref{formalSalemDef} and equation (\ref{SalemDef})) as a bound on the operator norm of the action on $W$, and to easily refer back to the global action when it becomes relevant in section \ref{ProductRings}. We fix our notations as follows:

Given $t\in R^\times,$ we denote by $A_t'$ the map from $V$ to itself given by:

$$(A_t'(f))(x) = \sum\limits_{y:y\cdot x = t}f(y)$$

and we let $A_t = A_t'|_{W}$ denote the operator $W\rightarrow V$ formed by restricting $A_t'$. NB: the image of $A_t$ is in general not contained in $W$. For instance, let $f$ be the function defined by:
$$f({x}) = \begin{cases}
    -1 & {x} = {0}\\
    \frac{1}{|R|^d-1} & {x} \neq {0}
\end{cases}$$
This function lies in $W$, yet $A_t f\not\in W$ as for each ${x} \in R^d$, $A_t(f)(x) \geq 0,$ as ${0}\cdot {x} \neq t$ no matter the value of ${x}$.

As in \cite{Kings} and \cite{IoMPak}, it will be convenient to define an appropriate Salem-number of the underlying ring $R$, which is the best constant $C$ for which $A_t$ satisfies the $C$-Salem bound (\ref{SalemDef}) over the ring $R$.
\begin{definition}
    Fix a ring $R$ and a unit $t\in R^\times$. The \textbf{$d$-dimensional $t$-incidence-Salem number} of $R$ is defined to be
    $$\frac{||A_t||_2}{|R|^\frac{d-1}{2}}$$
where the $L^2$ norm of $A_t$ is the operator norm of the action of $A_t$ on the space $W$ equipped with the $L^2$ norm.
\end{definition}
\begin{remark}
    By proposition \ref{allUnitsCreatedEqual}, this quantity is in fact independent of the choice of unit $t$. We nevertheless explicitly reference $t$ so that the underlying operator $A_t$ and equation
    $$y\cdot x = t$$
    are always clear from context.
\end{remark}
\subsection{\texorpdfstring{$A_t$ over Finite Fields}{A-t over Finite Fields}}
In section 3.2 of \cite{HartIoKohRud}, Hart, Iosevich, Koh, and Rudnev essentially computed $||A_tE||_2$ for $E$ the indicator function of a set $E$ in $\mathbb{F}_q^d;$ their argument generalizes to show that over a finite field $\mathbb{F}_q$ we have $||A_t||_2 \leq \sqrt{2}q^\frac{d-1}{2}$, so that the $d$-dimensional $t$-incidence-Salem number of an arbitrary finite field is at most $\sqrt{2}$. For convenience, we review the argument in our setting:
\begin{theorem}\label{finiteFieldsBound}
    Let $R = \mathbb{F}_q$, a finite field, $t\in\mathbb{F}_q^\times$, and $d\geq 2$. Then $$||A_t||_2\leq \sqrt{2}q^\frac{d-1}{2}.$$
\end{theorem}
\begin{proof}
    Let $\chi$ be a fixed nontrivial primitive additive character of $\mathbb{F}_q$, so that $\chi(sx)$ runs through all the characters of $\mathbb{F}_q$ as $s$ runs through the elements of $\mathbb{F}_q$.
    
    For a fixed $x$, we have that
    \[
    \begin{split}
        A_tf(x) &= \Slim{y\cdot x = t}f(y)\\
        &= \frac{1}{q}\Slim{y}\Slim{s}f(y)\chi(s(y\cdot x-t))\\
        &=\frac{1}{q}\Slim{y}f(y) + \frac{1}{q}\Slim{y}\Slim{s\neq 0}f(y)\chi(s(y\cdot x-t))\\
        &= \frac{1}{q}\Slim{y}\Slim{s\neq 0}f(y)\chi(s(y\cdot x-t))
    \end{split}
    \]
    where the last equality follows from the fact that $f$ is a mean-zero function. We therefore have
    \[
    \begin{split}
        ||A_tf||_2^2 &= \Slim{x}\left\lvert A_t f(x)\right\rvert^2\\
        &= \Slim{x}\frac{1}{q^2}\Slim{y, y'}\Slim{s, s'\neq 0}f(y)\overline{f(y')}\chi(s(y\cdot x-t))\overline{\chi(s'(y'\cdot x - t))}\\
        &= q^{-2}\Slim{y, y'}\Slim{s, s'\neq 0}f(y)\overline{f(y')}\Slim{x}\chi((sy-s'y')\cdot x + t(s'-s))\\
        &= q^{d-2}\Slim{\substack{sy = s'y'\\s, s' \neq 0}}\chi(t(s'-s))f(y)\overline{f(y')}\\
        &= I + II,
    \end{split}
    \]
    where the term $I$ corresponds to the terms where $s = s'$, while the term $II$ corresponds to the terms where $s\neq s'.$

    The first term is straightforward: as $s = s' \neq 0,$ $sy=s'y'$ implies that $y=y'$, so
    $$I = q^{d-2}\Slim{y, s\neq 0}|f(y)|^2 = (q-1)q^{d-2}||f||_2^2 \leq q^{d-1}||f||_2^2.$$

    For the second term, we make the substitution $a = s/s',$ $b = s'$ and write
    \[
    \begin{split}
        |II| &= q^{d-2}\left\lvert \Slim{y}\Slim{a\neq0, 1}f(y)\overline{f(ay)}\Slim{b\neq 0}\chi(tb(1-a))\right\rvert\\
        &= \left\lvert -q^{d-2}\Slim{a\neq0, 1}\Slim{y}f(y)\overline{f(ay)}\right\rvert\\
        &\leq q^{d-2}\Slim{a\neq0, 1}||f(y)||_2||f(ay)||_2\\
        &\leq q^{d-1}||f||_2^2,
    \end{split}
    \]
    using the Cauchy-Schwarz inequality to get from the second to the third line. It follows that $||A_tf||_2^2\leq 2q^{d-1}||f||_2^2,$ so that $||A_t||_2\leq \sqrt{2}q^\frac{d-1}{2}.$
\end{proof}
\subsection{The Structure of Finite Rings}
An essential ingredient in our theorem is the following structure theorem for finite rings (see proposition 2.5 of \cite{IoMPak}):
\begin{theorem}
    Every finite ring $R$ contains a two-sided ideal $J$, known as the \textit{Jacobson radical}, such that the quotient ring $R/J$ is isomorphic to a finite product of matrix rings over finite fields (possibly of dimension 1, i.e., finite fields).
\end{theorem}
We note that finite products of finite matrix rings are exactly the finite \textit{semisimple} rings (see section C-2.4 of \cite{Rotman} for a discussion of semisimple rings, and proposition 2.5 of \cite{IoMPak} for a proof of this fact). The structure of finite rings will guide the structure of our proof: in section \ref{Matrices} we will prove our theorem in the case where $R$ is a matrix ring; in section \ref{ProductRings} we will extend to finite products of matrix rings, and in section \ref{Rings} we will finally prove our theorem for all rings. 
\section{Matrix Rings}\label{Matrices}
As discussed in section \ref{solvabilitySection}, when $R$ has a sufficiently large dimension\footnote{In general, for $R = \text{Mat}_{n\times n}(\mathbb{F}_q)$, $n> 4$ is good enough. When $d \geq 4$, $n> 2$ is good enough.}, there is a conceptual proof that $A_t$ cannot satisfy the $C$-Salem bound (\ref{SalemDef}) using solvability of equations over subsets. We outline this argument in section \ref{Solvability} below. First, we give a computational argument which handles all cases.
\begin{theorem}\label{matrixRings}
    Let $R = \text{\emph{Mat}}_{n\times n}(\mathbb{F}_q)$, $t\in R^\times,$ and $d\geq 2$. There is a nontrivial character $\chi$ such that 
    $$\frac{||A_t\chi||_2}{|R|^{\frac{d-1}{2}}||\chi||_2} \geq \frac{1}{2}q^\frac{(n^2-n)(d-1)}{2}.$$
    In particular, for any fixed $C$, there are only finitely many non-field matrix rings whose $t'$-incidence-Salem number is bounded above by $C$ for some (equivalently, any) unit $t'\in R^\times$.

\end{theorem}
\begin{proof}
    Let $R$ be an $n\times n$ matrix ring over the field with $q$ elements, and fix $t\in R^\times.$ Our strategy is to manually construct a nontrivial character of $R^d$ which is constant on hyperplanes
    $$y\cdot x = t$$
    for a positive proportion of choices of $x$. Let $\chi_F$ be a nontrivial character of the underlying field. Let $\chi$ be the character of $R^d$ defined by 
    $$\chi(A_1, A_2, \dots, A_d) = \chi_F(a_1+a_2+\dots+a_d)$$
    where $a_i$ is the upper left entry of $A_i$. As $\chi_F$ is a nontrivial character, so is $\chi$, from which (by orthogonality with the trivial character) we have $\chi\in W$. We note that

$$||\chi||_2 = \sqrt{\sum\limits_{(A_1, A_2,\dots, A_d)\in R^d}|\chi(A_1,\dots, A_d)|^2} = \sqrt{\sum\limits_{{A}\in R^d}1} =q^{n^2d/2}.$$

To estimate $||A_t||_2$ we will estimate $(A_t\chi)(x)$ for $x$ in a special family of $d$-tuples of matrices. Let $S$ be the set of all ${x} = (x_1, x_2,\dots, x_d)\in R^d$ such that (1) $x_1$ is invertible, and (2) for each $2\leq i \leq d$, the rows of the matrix $x_i'$ formed by subtracting the first row of $x_1$ from the first row of $x_i$ are all contained in the span of the lower $n-1$ rows of $x_1$.\footnote{In the trivial case $n=1$, this corresponds to requiring $0\neq x_1=x_2=\dots = x_d$. This is wasteful but gives this theorem a uniform statement that interfaces simply with the rest of the paper.} Clearly,
$$|S| = |R^\times|q^{(n^2-n)(d-1)}$$
(pick a unit for $x_1$; the span of its last $n-1$ rows has size $q^{n-1}$ as all its rows are independent, and then we pick $n$ vectors from this span to construct the rows of $x_i'$ for each $2\leq i \leq d$). By lemma 2.6 in \cite{IoMPak}, $|R^\times| \geq \frac{1}{4}|R|$, so that 
$$|S|\geq \frac{1}{4}q^{(n^2-n)(d-1) + n^2}.$$ 

For each fixed ${x}\in S$ and $t\in R$, ${y}\cdot {x} = t$ has precisely $q^{n^2(d-1)}$ solutions, given by freely varying $y_2,\dots, y_d$. 
 Due to our construction of $\chi$ and $S$, $\chi(y)$ will be constant over this set of solutions. Indeed, for each $2\leq i \leq d$, $x_ix_1^{-1}$ has $(1, 0, \dots, 0)^T$ as its first column, as the first column of $x_1^{-1}$ is orthogonal to all of the last $(n-1)$ rows of $x_1$, and thus to any vector in their span, including every row of $x_i'.$ 
Thus, the first column of $y_ix_ix_1^{-1}$ is the same as the first column of $y_i$. We have that 
$$y_1x_1 = t-y_2x_2 - y_3x_3 - \dots - y_dx_d$$
so 
$$y_1 = tx_1^{-1}-y_2x_2x_1^{-1} - y_3x_3x_1^{-1} - \dots - y_dx_dx_1^{-1}$$
and 
$$y_1 +y_2x_2x_1^{-1} + y_3x_3x_1^{-1} + \dots + y_dx_dx_1^{-1} = tx_1^{-1}.$$
As the first column of $y_ix_ix_1^{-1}$ is the same as the first column of $y_i$, we have that 
$$(y_1+\dots+y_d)_{1,1} = (tx_1^{-1})_{1,1},$$
and thus
$$\chi({y}) = \chi_F[(tx_1^{-1})_{1,1}].$$
It follows that for $x\in S$, we have that $A_t\chi(x) = q^{n^2(d-1)}\chi_F((tx_1^{-1})_{1, 1})$, so
$$||A_t\chi||_2 \geq |S|^\frac{1}{2}q^{n^2(d-1)} \geq \frac{1}{2}q^{((n^2-n)(d-1) + n^2)/2}q^{n^2(d-1)}$$
Thus, we have:
$$||A_t||_2 \geq \frac{1}{2}q^{((n^2-n)(d-1) + n^2)/2}q^{n^2(d-1)}q^{-n^2d/2},$$
so that the $t$-incidence-Salem number of $R$ is at least
$$\frac{1}{2}q^{((n^2-n)(d-1) + n^2)/2}q^{n^2(d-1)}q^{-n^2d/2}q^{-n^2(d-1)/2} = \frac{1}{2}q^\frac{(n^2-n)(d-1)}{2}.$$
Now, assume that the $t$-incidence-Salem number of $R$ is at most $C$. After taking (base-$q$) logs we get:

$$(n^2 - n)(d-1)\leq 2\log_q(2C).$$

For the finiteness statement, we restrict to the case of non-field matrix rings, so $n\geq 2.$ Because we also have $d\geq 2$, the left hand side of the inequality is clearly always at least 1, so $q$ is bounded above. Also, the left hand side increases with $n$, so that $n$ is bounded above for each $q$, giving us the desired finiteness.

\end{proof}
\subsection{Solvability of Constrained Equations}\label{Solvability}
    As mentioned above in subsection \ref{solvabilitySection}, we can prove a slightly weaker version of theorem \ref{matrixRings} using the techniques of \cite{IoMPak} and \cite{Kings} together with the following proposition, inspired by \cite{HartIoKohRud}.
    \begin{theorem}\label{EdotE}
        Let $R$ be a finite ring, and let $t\in R^\times.$ Let $N(R) = \#\{(x, y)\in R^d\times R^d| x\cdot y = t\}$. Suppose that $A_t$ over $R^d$ satisfies the $C$-Salem bound (\ref{SalemDef}), and suppose that $R$ is isomorphic to its opposite ring $R^\text{op}$, say by the map $\varphi$.\footnote{The \textit{opposite ring} $R^\text{op}$ is the ring formed from $R$ by reversing the order of multiplication. Concretely, $\varphi$ is a bijection $R\rightarrow R$ such that $\varphi(x+y) = \varphi(x) + \varphi(y)$ and $\varphi(xy) = \varphi(y)\varphi(x).$} Then for any $E\subseteq R^d$ such that 
            $$|E| > \frac{2C|R|^\frac{d-1}{2}|R|^{2d}}{N(R)},$$
        $t\in E\cdot E.$
    \end{theorem}
    \begin{proof}
        Let $\nu(t) = |\{(x, y)\in E\times E| x\cdot y = t\}|.$ By abuse of notation, let $E$ denote the indicator function of $E$. Then we have that 
        \[\begin{split}
 	      \nu(t) &= \Slim{x}\sum_{y:y\cdot x = t}E(x)E(y)\\
               &= \Slim{x}E(x)\sum_{y:y\cdot x = t}\frac{|E|}{|R|^d} + \Slim{x}E(x)\sum_{y:y\cdot x = t}\left(E(y) - \frac{|E|}{|R|^d}\right)\\
               &= \Slim{x}E(x)\sum_{y:y\cdot x = t}\frac{|E|}{|R|^d} + I
        \end{split}\]
where $I = \Slim{x}E(x)\sum_{y:y\cdot x = t}\left(E(y) - \frac{|E|}{|R|^d}\right).$ We estimate $I$ using Cauchy-Schwarz, the bound 
$$\Norm{E-\frac{|E|}{|R|^d}}_2 \leq ||E||_2,$$
and the $C$-Salem bound (\ref{SalemDef}) on $A_t$:
\[
\begin{split}
    |I| &= \Abs{\Slim{x}E(x)\sum_{y:y\cdot x = t}\left(E(y) - \frac{|E|}{|R|^d}\right)} \\
    &\leq \Norm{E}_2\Norm{A_t\left(E-\frac{|E|}{|R|^d}\right)}_2\\
    &\leq C|R|^\frac{d-1}{2}||E||_2^2\\
    &= C|R|^\frac{d-1}{2}|E|
\end{split}
\]
As for the first term, we rearrange and apply the coordinate transform $t' = \varphi(t)$, $x' = \varphi(y)$, and $y' = \varphi(x)$, where we have applied $\varphi$ coordinatewise to all vectors:
\[
\begin{split}
    \Slim{x}E(x)\sum_{y:y\cdot x = t}\frac{|E|}{|R|^d}&= \frac{|E|}{|R|^d}\Slim{y}\Slim{x:y\cdot x = t}E(x) \\
    &=\frac{|E|}{|R|^d}\Slim{x'}\Slim{y':y'\cdot x' = t'}E(\varphi^{-1}(y'))\\
    &= \frac{|E|}{|R|^d}\Slim{x'}\Slim{y':y'\cdot x' = t'}\frac{|E|}{|R|^d} + \frac{|E|}{|R|^d}\Slim{x'}\Slim{y':y'\cdot x' = t'} \left(E(\varphi^{-1}(y'))-\frac{|E|}{|R|^d}\right)\\
    &= \frac{|E|^2}{|R|^{2d}}N(R) + II
\end{split}\]
where $II = \frac{|E|}{|R|^d}\Slim{x'}\Slim{y':y'\cdot x' = t'} \left(E(\varphi^{-1}(y'))-\frac{|E|}{|R|^d}\right).$ We estimate $II$ much like $I$, using the fact that $A_{t'}$ also satisfies the $C$-Salem bound (\ref{SalemDef}) by proposition \ref{allUnitsCreatedEqual}:
\[
\begin{split}
    |II| &= \Abs{\frac{|E|}{|R|^d}\Slim{x'}\Slim{y':y'\cdot x' = t'} \left(E(\varphi^{-1}(y'))-\frac{|E|}{|R|^d}\right)}\\
    &\leq \frac{|E|}{|R|^d}|R|^\frac{d}{2}||A_{t'}||_2|E|^\frac{1}{2}\\
    &\leq C|R|^\frac{d-1}{2}\frac{|E|^{3/2}}{|R|^{d/2}};
\end{split}
\]
applying the trivial bound $|E|\leq |R|^d$ we have
$$|II| \leq C|R|^\frac{d-1}{2}|E|.$$
It follows from the above that
$$\Abs{\nu(t) - \frac{|E|^2}{|R|^{2d}}N(R)} \leq 2C|R|^\frac{d-1}{2}|E|.$$
As long as 
$$\frac{|E|^2}{|R|^{2d}}N(R) > 2C|R|^\frac{d-1}{2}|E|$$
we have $\nu(t)>0$; this occurs so long as 
$$|E| > \frac{2C|R|^\frac{d-1}{2}|R|^{2d}}{N(R)}.$$
    \end{proof}
Typically, we expect $N(R)$ to grow like $c|R|^{2d-1}$, as the condition $x\cdot y = t$ should cut out a $1$-codimensional subset of $R^{2d}$. We remark on two special cases:
\begin{proposition}\label{N(R)}
Over finite fields, we have that $N(\mathbb{F}_q) = q^{2d-1} - q^{d-1}$. Over matrix rings, we have the uniform bound $N(\text{Mat}_{n\times n}(\mathbb{F}_q)) \geq \frac{1}{4}|R|^{2d-1}$.
\end{proposition}
\begin{proof}
Over finite fields, we simply compute using linear algebra: for each ${x}\neq {0},$ there are precisely $q^{d-1}$ solutions to ${x}\cdot{y} = t$, whereas for  ${x} = {0}$, there are no solutions to ${x}\cdot{y} = t$ as $t \neq 0$. ${x}$ can take on  $q^d-1$ nonzero values, so that $N(\mathbb{F}_q) = q^{d-1}(q^d - 1) = q^{2d-1} - q^{d-1}.$

Over matrix rings, there are more non-units to keep track of, so that we will settle for just a lower bound. If $x_1\in R^\times$, then ${x}\cdot{y} = t$ has precisely $|R|^{d-1}$ solutions, given by freely varying $y_2, \dots, y_d$ and solving for $y_1$ using the fact that $x_1$ is a unit. By lemma 2.6 in \cite{IoMPak}, $|R^\times|\geq \frac{1}{4}|R|$, so that there are $\geq \frac{1}{4}|R|^d$ possible values for ${x}$, and thus $N(R)\geq \frac{1}{4}|R|^{2d-1}$.  
\end{proof}
The following corollary is immediate from theorem \ref{finiteFieldsBound}, theorem \ref{EdotE}, proposition \ref{N(R)}, and the fact that any matrix ring is isomorphic to its opposite by the map taking a matrix to its transpose:
\begin{corollary}
Given $E\subseteq \mathbb{F}_q^d$, if $|E|> \frac{2\sqrt{2}}{1-q^{-d}}q^\frac{d+1}{2}$, $E\cdot E\supseteq\mathbb{F}_q^\times$.

Suppose that $A_t$ satisfies the $C$-Salem bound (\ref{SalemDef}) over a matrix ring $R$. If $E\subseteq R^d$ satisfies $|E|> 8C|R|^\frac{d+1}{2}$, then $t\in E\cdot E$.
\end{corollary}
In the finite field setting, our bounds are not as strong as they could be as a consequence of the level of generality at which we work; corollary 2.4 of \cite{HartIoKohRud} gives a version of the above result with the factor of $\frac{2\sqrt{2}}{1-q^{-d}}$ replaced by a factor of $1$. The key ingredient there not present here is the possibility of discarding some of the terms arising in the proof of our theorem \ref{finiteFieldsBound} by positivity, which is not possible when we work at the level of generality in which we state theorem \ref{finiteFieldsBound}.

The key input to the techniques of \cite{Kings} and \cite{IoMPak} is the following fact:
\begin{proposition}\label{idealBound}
Suppose $R$ is a ring and $t$ is a unit of $R$ such that there is a constant $C_1$ such that any $E\subseteq R^d$ satisfying 
$$|E|> C_1|R|^\frac{d+1}{2}$$
has the property that $t\in E\cdot E$. Then for any proper left (or right) ideal $I\subseteq R$ we have that
$$|I| \leq C_1^\frac{1}{d}|R|^{\frac{1}{2} + \frac{1}{2d}}.$$
In particular, if $R$ is isomorphic to its opposite ring, and there are constants $C, c$ such that $N(R)\geq c|R|^{2d-1}$ and $A_t$ satisfies the $C$-Salem bound (\ref{SalemDef}), then for any proper left (or right) ideal $I\subseteq R,$ we have that
$$|I| \leq (2C/c)^\frac{1}{d}|R|^{\frac{1}{2} + \frac{1}{2d}}.$$
\end{proposition}
\begin{proof}
Let $I$ be a proper left (or right) ideal, and let $E = I^d$. We claim $t\notin E\cdot E$. If not, then we may write
$$t = x_1y_1 + \dots + x_dy_d$$
with $x_1, x_2, \dots, x_d, y_1, y_2, \dots, y_d\in I$. As $I$ is an ideal (left or right) it follows that $t\in I$. But as $t$ is a unit, it follows that $I = R,$ contradicting properness of $I$. It follows that 
$$|E|\leq C_1|R|^\frac{d+1}{2}$$
and thus that 
$$|I| \leq C_1^\frac{1}{d}|R|^{\frac{1}{2} + \frac{1}{2d}}.$$
\end{proof}
To establish our result for matrix rings of large enough dimension, one can play this bound on the size of ideals off of the known size of the (left) ideal consisting of all matrices with final column ${0}$. The exact cutoff for the dimensions of the matrix rings handled depends on $d$: if $d\geq 4,$ we can rule out matrix rings of dimension $n > 2$; for $d = 3$ we can only rule out matrix rings of dimension $n > 3$; for $d = 2$ we can only rule out matrix rings of dimension $n > 4$. See the proof of theorem 2.15 in \cite{Kings} for the details (in \cite{IoMPak}, the special form of the relevant equation allows the authors to squeeze more out of these techniques, and the numerology is slightly different). 
\section{Semisimple Rings}\label{ProductRings}
As in \cite{IoMPak} and \cite{Kings}, in order to extend our results to semisimple rings (that is, products of matrix rings), we will need three ingredients: a lower bound on $||A_t'||_2$ provided by computing $||A_t'\chi_0||_2$; a way to factor $||A_t||_2$ over a product ring; and a uniform lower bound on the $t$-incidence-Salem number of any matrix ring $R$, including fields. We start with the last ingredient, as we've already done most of the work in theorem \ref{matrixRings}.
\begin{lemma}\label{generalLower}
    Let $R = \text{\emph{Mat}}_{n\times n}(\mathbb{F}_q)$, $n\geq 1$, and $t\in R^\times.$ Then there is a nontrivial character $\chi$ of $R^d$ such that 
    $$\frac{||A_t\chi||_2}{|R|^{\frac{d-1}{2}}||\chi||_2}\geq 1/2.$$ 
    In particular, the $t$-incidence-Salem number of $R$ is at least $\frac{1}{2}.$
\end{lemma}
\begin{proof}
Let $\chi$ be the nontrivial character produced by theorem \ref{matrixRings}. Then  
$$\frac{||A_t\chi||_2}{|R|^{\frac{d-1}{2}}||\chi||_2}\geq \frac{1}{2}q^\frac{(n^2-n)(d-1)}{2}.$$
As $d\geq 2$ and $n\geq 1,$ $(n^2 - n)(d-1)\geq 0,$ so $\frac{1}{2}q^\frac{(n^2-n)(d-1)}{2}\geq \frac{1}{2}.$

\end{proof}

Now, suppose the ring $R$ factors as $R\cong R_1\times R_2$ (by abuse of notation, we'll identify the two) with $t = (t_1, t_2)$, $t_j\in R_j^\times$, $j = 1, 2$. By $V_j$ we mean the vector space of complex-valued functions on $R_j^d,$ and by $W_j$ we mean the subspace of \textit{mean-zero} complex-valued functions on $R_j^d,$ $j = 1,2$. $A_{t_1}'$ and $A_{t_2}'$ will be the relevant summation operators on $V_1$ and $V_2$, respectively. Any character $\chi$ of $R^d$ will factor as $\chi_1\chi_2$, where $\chi_j$ is a character of $R_j^d$. For any function $f\in V$ which factors as $f_1f_2$, $f_j\in V_j$, we have that $f\in W$ if at least one $f_j\in W_j$; in particular, $\chi$ is a nontrivial character so long as at least one of the characters $\chi_j$ is nontrivial.

\begin{lemma}\label{Products}
    In the above scenario, we have $||A_t'\chi||_2 = ||A_{t_1}'\chi_1||_2||A_{t_2}'\chi_2||_2$.
\end{lemma}
\begin{proof}
    We have that

$$\sum\limits_{y\cdot x = t} \chi(y) = \sum\limits_{y_1\cdot x_1 = t_1;\, y_2\cdot x_2 = t_2} \chi_1(y_1)\chi_2(y_2) = \left(\sum\limits_{y_1\cdot x_1 = t_1}\chi_1(y_1)\right)\left(\sum\limits_{y_2\cdot x_2 = t_2} \chi_2(y_2)\right),$$

where $x = (x_1, x_2),$ and $ y = (y_1, y_2)$ are decompositions of $x$ and $y$ into elements of $R_1^d$ and $R_2^d$. As a result, for $x = (x_1, x_2)$ and $t =(t_1, t_2)$, we may write

$$A_{t}'\chi(x) = (A_{t_1}'\chi_1(x_1))(A_{t_2}'\chi_2(x_2))$$

Taking norms, squaring, summing over $x$, and taking a square root, we have that

$$||A_t'\chi||_2 = ||A_{t_1}'\chi_1||_2||A_{t_2}'\chi_2||_2$$

\end{proof}
In particular, we can choose either of $\chi_1$ or $\chi_2$ to be the trivial character and still have $\chi\in W$. As there will be no cancellation in the sum defining $A_{t_j}'(\chi_j)$ if $\chi_j$ is trivial, we expect this to make $||A_t'\chi||_2$ large, so that $R$ cannot have square-root cancellation. To make this precise, we derive a lower bound on $A_t'\chi$ for $\chi = \chi_0,$ the trivial character:
\begin{lemma}\label{trivialChar}
    Let $R$ be a ring, and $t\in R^\times.$ Then 
    $$\frac{||A_t'\chi_0||_2}{||\chi_0||_2} \geq |R^\times||R|^{d - 2}.$$
    In particular, if $R$ is a matrix ring, 
    $$\frac{||A_t'\chi_0||_2}{||\chi_0||_2} \geq \frac{1}{4}|R|^{d-1}.$$
\end{lemma}
\begin{proof}
    We compute:

$$||A_t'\chi_0||_2^2 = \sum\limits_{{x}}\left(\sum\limits_{{y}: {y}\cdot {x} = t} 1\right)^2 = \sum\limits_{{x},\, {y}:\, {y}\cdot {x} = t}\;\sum\limits_{{y'}:\, {y'}\cdot {x} = t} 1$$

We rewrite the sum as:

$$\sum\limits_{{y}}\sum\limits_{{x}: {y}\cdot {x} = t}\sum\limits_{{y'}:({y}-{y'})\cdot {x} = 0} 1 \geq \sum\limits_{{y}: y_1\in R^\times}\sum\limits_{\substack{{x}: {y}\cdot {x} = t,\\ x_2\in R^\times}}\sum\limits_{y': (y-y')\cdot x = 0} 1$$

For any fixed $x$ and $y$, as $x_2$ is a unit, for a fixed choice of $y'_1, y'_3, y'_4, \dots$, there is precisely one choice of $y'_2$ such that $(y_1-y'_1)x_1 + (y_2-y'_2)x_2 + \dots + (y_d - y'_d)x_d = 0$, so the innermost sum is $|R|^{d-1}$. For any fixed $y$, as $y_1$ is a unit, there are similarly precisely $|R^\times||R|^{d-2}$ possible choices of $x$, and finally there are $|R^\times||R|^{d-1}$ possible values of $y$, so all in all, we have that:

$$||A_t'\chi_0||_2^2 \geq |R^\times|^2|R|^{3d-4}$$

so that (recalling $||\chi_0||_2 = |R|^{d/2}$)

$$\frac{||A_t'\chi_0||_2}{||\chi_0||_2} \geq |R^\times||R|^{d - 2}$$.

In particular, if $R$ is a matrix ring, we have that $|R^\times|/|R| \geq \frac{1}{4}$ by lemma 2.6 of \cite{IoMPak}, and the bound becomes:

$$\frac{||A_t'\chi_0||_2}{||\chi_0||_2} \geq \frac{1}{4}|R|^{d-1}.$$

\end{proof}

This is a good enough lower bound on $\frac{||A_t'\chi_0||_2}{||\chi_0||_2}$ for all simple rings but one: if $R = \mathbb{F}_2$, there is only one unit, and for small values of $d$, the lower bound we have will not suffice. Fortunately, $\mathbb{F}_2$ is simple enough that we can compute $||A_{1}'\chi_0||_2$ exactly:
\begin{lemma}\label{BooleanRings}
    Let $R = \mathbb{F}_2$. Then 
    $$\frac{||A_{1}'\chi_0||_2}{||\chi_0||_2} = 2^{\frac{d}{2}-1}(2^d-1)^\frac{1}{2}.$$
\end{lemma}
\begin{proof}
    First we compute $A_{1}'\chi_0({x}).$ If ${x} = {0},$ then there are no ${y}$ such that $y\cdot x = 1$. Thus, assume $x$ contains precisely $k$ $1$'s, $k\geq 1$. To make ${y}\cdot {x} = 1,$ we must choose an odd number of the $k$ indices in which ${x}$ is nonzero to put a ``1'' in $y$, and we may assign the remaining $d-k$ indices of $y$ freely. The number of ways to choose an odd number of elements of a $k$-element set is $2^{k-1}$ by a classic bit of combinatorics:

\[ \begin{split}
        \sum\limits_{\substack{0\leq \ell\leq k\\ \ell\equiv 1\text{ (mod 2)}}} \binom{k}{\ell} &= \frac{1}{2}\sum\limits_{0\leq \ell \leq k}(1^\ell - (-1)^\ell)\binom{k}{\ell} \\
&= \frac{1}{2} \left(\sum\limits_{0\leq \ell\leq k}\binom{k}{\ell}1^\ell - \sum\limits_{0\leq \ell\leq k}\binom{k}{\ell}(-1)^\ell\right)  \\
&= \frac{1}{2}\left((1+1)^k - (1-1)^k\right)  \\
&= 2^{k-1}
    \end{split} \]

Thus, for a fixed choice of ${x}$ having $k\geq1$ $1$'s, there are exactly $2^{k-1}2^{d-k} = 2^{d-1}$ choices for ${y}$ so that ${y}\cdot{x} = 1,$ and for such an ${x}$, we have $(A_1'\chi_0)({x}) = \sum_{{y}\cdot{x} = 1}1 = 2^{d-1}.$ Now, there are exactly $\binom{d}{k}$ vectors ${x}$ containing exactly $k$ 1's, so that:

\[\begin{split}
    ||A_1'\chi_0||_2 &= \left(\sum\limits_{{x}}|(A_1'\chi_0)({x})|^2\right)^{1/2}\\
    &= \left(\sum_{1\leq k\leq d}\binom{d}{k}\left(2^{d-1}\right)^2\right)^{1/2}\\
    &= 2^{d-1}\left(\left(\sum\limits_{0\leq k\leq d}\binom{d}{k}\right) - 1\right)^{1/2}\\
    &= 2^{d-1}\sqrt{2^d-1}
\end{split}\]
Finally, as $||\chi_0||_2 = |R|^{d/2} = 2^{d/2},$ we find that:

$$\frac{||A_{1}'\chi_0||_2}{||\chi_0||_2} = 2^{\frac{d}{2}-1}(2^d-1)^\frac{1}{2}$$
\end{proof}

We are now able to extend our result to semisimple rings:
\begin{theorem}\label{semisimple}
    Fix a constant $C>0$, and an integer $d\geq 2$. There are only finitely many pairs $(R, t)$ with $R$ a finite semisimple ring and $t\in R^\times$ where $R$ is not a finite field and the $d$-dimensional $t$-incidence-Salem number of $R$ is at most $C$.
\end{theorem}
\begin{proof}
    Fix a ring $R$, and a unit $t\in R^\times$. As $R$ is semisimple, we may write 
    $$R = R_1\times R_2\times\dots\times R_m$$
    where each $R_j$ is simple, and thus is a matrix ring, i.e., of the form $\text{Mat}_{n_j\times n_j}(\mathbb{F}_{q_j})$ for some $n_j\geq 1$ and finite field $\mathbb{F}_{q_j}$ with $q_j$ elements. If $m = 1,$ then $R$ is in fact simple, and this is just theorem \ref{matrixRings}. Thus, suppose $m\geq 2$. Let $t = (t_1, t_2, \dots, t_m)$ be a decomposition of $t$ into elements of $R_1,\,R_2,\,\dots,\,R_m$. Let $\chi = \chi_1\prod_{2\leq j \leq m}\chi_{R_j, 0}$, where $\chi_1$ is the nontrivial character of $R_1^d$ given by lemma \ref{generalLower}, so that 
    $$\frac{||A_{t_1}'\chi_1||_2}{|R_1|^{\frac{d-1}{2}}||\chi_1||_2}\geq \frac{1}{2},$$
    and $\chi_{R_j, 0}$ is the trivial character of $R_j^d$. Clearly $\chi\in W$, so $A_t'\chi = A_t\chi.$ By repeated application of lemma \ref{Products}, we have that:

    \[
\begin{split}
    \frac{||A_t\chi||_2}{||\chi||_2} &= \frac{||A_{t_1}'\chi_1||_2||A_{t_2}'\chi_{R_2, 0}||_2\cdots||A_{t_m}'\chi_{R_m, 0}||_2}{|R|^{d/2}}\\
    &= \frac{||A_{t_1}'\chi_1||_2}{|R_1|^{d/2}}\frac{||A_{t_2}'\chi_{R_2, 0}||_2}{|R_2|^{d/2}}\cdots\frac{||A_{t_m}'\chi_{R_m, 0}||_2}{|R_m|^{d/2}}\\
    &=\frac{||A_{t_1}'\chi_1||_2}{||\chi_1||_2}\frac{||A_{t_2}'\chi_{R_2, 0}||_2}{||\chi_{R_2, 0}||_2}\cdots\frac{||A_{t_m}'\chi_{R_m, 0}||_2}{||\chi_{R_m, 0}||_2}.
\end{split}    
    \]
    In particular, the $t$-incidence-Salem number of $R$ is bounded below by
    \[
        \frac{||A_t\chi||_2}{|R|^{\frac{d-1}{2}}||\chi||_2} = \frac{||A_{t_1}'\chi_1||_2}{|R_1|^{\frac{d-1}{2}}||\chi_1||_2}\frac{||A_{t_2}'\chi_{R_2, 0}||_2}{|R_2|^{\frac{d-1}{2}}||\chi_{R_2, 0}||_2}\cdots\frac{||A_{t_m}'\chi_{R_m, 0}||_2}{|R_m|^{\frac{d-1}{2}}||\chi_{R_m, 0}||_2}
    \]

Applying the lower bound 
$$\frac{||A_{t_1}'\chi_1||_2}{|R_1|^{\frac{d-1}{2}}||\chi_1||_2}\geq \frac{1}{2}$$
along with our assumption on the $t$-incidence-Salem number of $R$, we have:
$$\frac{||A_{t_2}'\chi_{R_2, 0}||_2}{|R_2|^{\frac{d-1}{2}}||\chi_{R_2, 0}||_2}\cdots\frac{||A_{t_m}'\chi_{R_m, 0}||_2}{|R_m|^{\frac{d-1}{2}}||\chi_{R_m, 0}||_2} \leq 2C.$$

We now derive lower bounds for each factor
$$\frac{||A_{t_j}'\chi_{R_j, 0}||_2}{|R_j|^{\frac{d-1}{2}}||\chi_{R_j, 0}||_2}.$$

Suppose $R_j$ ($2\leq j \leq m$), $|R_j| \geq 17.$ Then by lemma \ref{trivialChar}, we have that 
$$\frac{||A_{t_j}'\chi_{R_j, 0}||_2}{|R_j|^{\frac{d-1}{2}}||\chi_{R_j, 0}||_2}\geq\frac{1}{4}|R_j|^\frac{d-1}{2} \geq \frac{\sqrt{17}}{4} > 1.$$

If $|R_j|<17$, then either $R_j$ is a field, or $R_j = \text{Mat}_{2\times2}(\mathbb{F}_2).$ In the latter case, one can count by hand that $|R_j^\times| = |GL_2(\mathbb{F}_2)| = 6$. Then by lemma \ref{trivialChar} we have:
$$\frac{||A_{t_j}'\chi_{R_j, 0}||_2}{|R_j|^{\frac{d-1}{2}}||\chi_{R_j, 0}||_2} \geq |R_j^\times||R_j|^{\frac{d-3}{2}} \geq \frac{6}{4} > \frac{\sqrt{17}}{4} > 1.$$
Suppose then $R_j$ is a field, so $|R_j^\times| = |R_j|-1.$ If $|R_j|\geq3,$ then by lemma \ref{trivialChar} we have:
$$\frac{||A_{t_j}'\chi_{R_j, 0}||_2}{|R_j|^{\frac{d-1}{2}}||\chi_{R_j, 0}||_2} \geq |R_j^\times||R_j|^{\frac{d-3}{2}} = (|R_j|-1)|R_j|^{\frac{d-3}{2}}\geq \sqrt{|R_j|} - \frac{1}{\sqrt{|R_j|}} \geq \frac{2\sqrt{3}}{3} > \frac{\sqrt{17}}{4}.$$
Finally, if $R_j = \mathbb{F}_2,$ then by lemma \ref{BooleanRings} we have:
$$\frac{||A_{t_j}'\chi_{R_j, 0}||_2}{|R_j|^{\frac{d-1}{2}}||\chi_{R_j, 0}||_2} = \frac{2^{\frac{d}{2}-1}(2^d-1)^\frac{1}{2}}{2^\frac{d-1}{2}} = \left(\frac{2^d-1}{2}\right)^\frac{1}{2}\geq \sqrt{\frac{3}{2}} > \frac{\sqrt{17}}{4}. $$

In particular, we have that
$$\left(\frac{\sqrt{17}}{4}\right)^{m-1}<2C,$$
so that $m$ is bounded. By the above case analysis, each factor is at least 1, whence for each $R_j$ we also have,

$$2C \geq \frac{||A_{t_j}'\chi_{R_j, 0}||_2}{|R_j|^{\frac{d-1}{2}}||\chi_{R_j, 0}||_2}\geq\frac{1}{4}|R_j|^\frac{d-1}{2}$$

so

$$|R_j| \leq (8C)^\frac{2}{d-1}.$$

Thus, there are only finitely many possible choices for each $R_j$ ($2\leq j\leq m$). Interchanging the roles of $R_1$ and $R_2$ in the above argument yields that there are only finitely many possible choices for $R_1$ as well, so that there are only finitely many possible non-simple semisimple rings $R$ whose $t$-incidence-Salem number is at most $C$. As each ring is finite, for each choice of $R$, there are only finitely many choices for $t$ as well, establishing the theorem.
\end{proof}
We also note the following fact, which is a corollary of the proof of theorem \ref{semisimple} and which will be used when we extend our result to the general case:
\begin{corollary}\label{lowerBoundSemisimple}
Let $R$ be a finite semisimple ring, and $t\in R^\times.$ Then there is a character $\chi$ of $R^d$ such that $$\frac{||A_t\chi||_2}{|R|^{\frac{d-1}{2}}||\chi||_2}\geq 1/2.$$ 
\end{corollary}
\begin{proof}
    Let $\chi$ be the character produced by the construction in the proof of theorem \ref{semisimple}. Then we may write:

    $$\frac{||A_t\chi||_2}{|R|^{\frac{d-1}{2}}||\chi||_2}=\frac{||A_{t_1}'\chi_1||_2}{|R_1|^{\frac{d-1}{2}}||\chi_1||_2}\frac{||A_{t_2}'\chi_{R_2, 0}||_2}{|R_2|^{\frac{d-1}{2}}||\chi_{R_2, 0}||_2}\cdots\frac{||A_{t_m}'\chi_{R_m, 0}||_2}{|R_m|^{\frac{d-1}{2}}||\chi_{R_m, 0}||_2}$$

    where we have that 

    $$\frac{||A_{t_1}'\chi_1||_2}{|R_1|^{\frac{d-1}{2}}||\chi_1||_2}\geq \frac{1}{2}$$
    and each
    $$\frac{||A_{t_j}'\chi_{R_j, 0}||_2}{|R_j|^{\frac{d-1}{2}}||\chi_{R_j, 0}||_2}\geq 1$$
    by the proof of theorem \ref{semisimple}. The result follows.     
\end{proof}
\section{General Finite Rings}\label{Rings}
In order to extend our results to general rings, we will need to understand how the presence of a nontrivial Jacobson radical influences $||A_t||_2.$ The key bound is the following:
\begin{theorem}\label{JacobsonBound}
    Let $R$ be a finite ring, and let $J$ be its Jacobson radical. Let $t\in R^\times$. Let $\tilde{\chi}$ be a nontrivial character of $(R/J)^d$, and let $\chi$ be the pullback to $R^d$. Then

    $$\frac{||(A_t\chi)||_2}{||\chi||_2} \geq |J|^{d-1}\frac{||(A_{\tilde{t}}\tilde\chi)||_2}{||\tilde\chi||_2}$$
\end{theorem}
\begin{proof}
    To estimate $||A_t\chi||_2,$ we will estimate $|A_t\chi({x})|$ for various choices of ${x}$. Fix ${x}\in R^d$ and let ${\tilde{x}}$ be the reduction of ${x}$ to $(R/J)^d,$ $\tilde{t}$ be the reduction of $t$ to $R/J$, and write:

$$A_t\chi(x) = \sum\limits_{{y}: {y}\cdot {x} = t}\chi({y}) = \sum\limits_{{\tilde{y}}\cdot {\tilde{x}} = \tilde{t}}\tilde{\chi}({\tilde{y}})\sum\limits_{\substack{{y}: y \equiv {\tilde{y}} \text{ (mod }J)\\ \ {y}\cdot {x} = t}}1$$

so that we must estimate the number of ${y}$ such that ${y} \equiv {\tilde{y}}$ and ${y}\cdot {x} = t.$ To do so, we introduce four additive homomorphisms (which may not be ring homomorphisms, but which are homomorphisms of the underlying abelian groups): $\varphi_{{x}}: R^d\rightarrow R$ given by $\varphi_{{x}}({y}) = {y}\cdot {x}$, $\tilde{\varphi}_{{\tilde x}}: (R/J)^d \rightarrow R/J$ by $\tilde{\varphi}_{{\tilde x}}({\tilde{y}}) = {\tilde{y}}\cdot {\tilde{x}}$, the projection $p$ from $R$ to $R/J$, and the $d$-dimensional projection $p^d$ from $R^d$ to $(R/J)^d.$ When the dependence on ${x}$ and ${\tilde x}$ is clear, we denote $\varphi_{{x}}$ simply by $\varphi$ and $\tilde{\varphi}_{{\tilde x}}$ simply by $\tilde{\varphi}.$ Of course, for each ${x}$ the following diagram commutes:
\[ \begin{tikzcd}
R^d \arrow{r}{\varphi} \arrow[swap]{d}{p^d} & R \arrow{d}{p} \\%
(R/J)^d \arrow{r}{\tilde{\varphi}}& R/J
\end{tikzcd}
\]
We note that the image of $\varphi$ (resp. $\tilde{\varphi}$) is the (left) ideal generated by $x_1,\dots, x_d$ (resp. $\tilde{x_1}\dots, \tilde{x_d}$) in $R$ (resp. $R/J$), so that $\varphi$ (resp. $\tilde{\varphi}$) is surjective if and only if its image contains units in $R$ (resp. $R/J$). Hence, for a given ${x}$, either the sum defining $A_t$ is empty, or the map $\varphi$ is surjective. 

In order to proceed, we need a fact about the Jacobson radical, which may be found in \cite{JacobsonFacts}, proposition 4.1 or \cite{Rotman}, proposition C-2.12:
\begin{lemma}\label{JacobsonFacts}
Let $R$ be a ring, and let $J$ be its Jacobson radical, and let $s\in R$. Then the following are equivalent:
\begin{enumerate}
\item $s\in J.$
\item For any $r\in R$, $1+rs$ is a unit.
\item For any $r \in R$, $1 + sr$ is a unit.

\end{enumerate}
\end{lemma}
This lemma will allow us to relate $\varphi$ with $\tilde{\varphi}$. For example, we can quickly prove the following result:
\begin{lemma}
$\varphi$ is surjective if and only if $\tilde{\varphi}$ is.
\end{lemma}
\begin{proof}
Clearly, surjectivity of $\varphi$ implies surjectivity of $\tilde{\varphi}.$ To establish the reverse, we will use lemma \ref{JacobsonFacts}. If $\tilde\varphi$ is surjective, then there is a $y\in R^d$ such that $\varphi(y) \equiv 1 \text{ (mod }J)$. Then $y\cdot x = 1+s$ for some $s\in J$, and as $1+s$ is a unit, the image of $\varphi$ contains units, which implies its surjectivity. 
\end{proof}
We will actually need the following slightly stronger result:

\begin{lemma}\label{congruenceConditions}
Fix a unit $t$ of $R$, and let $\tilde{t}$ denote its image in $R/J$. Fix $x\in R^d$, and let $\tilde{x}$ be its image in $(R/J)^d.$ Fix ${\tilde{y}}\in (R/J)^d$ so that ${\tilde{y}}\cdot {\tilde{x}} = \tilde{t}.$ Then there exists ${y}\in R^d$ so that ${y}\equiv {\tilde{y}}$ (mod $J$), and ${y} \cdot {x}= t$.
\end{lemma}
\begin{proof}
Fix a preimage ${y}'$ of ${\tilde{y}}.$ We note that ${y}'\cdot {x} = t + s$ for some $s\in J$, so that $(t^{-1}{y}')\cdot {x} = 1 + t^{-1}s.$  By lemma \ref{JacobsonFacts}, $1 + t^{-1}s$ is a unit, so that $(t(1+t^{-1}s)^{-1}t^{-1}{y}')\cdot x = t;$ we let ${y} = t(1+t^{-1}s)^{-1}t^{-1}{y}'.$ 

By construction, ${y}\cdot {x} = t;$ we must only show ${y}\equiv {\tilde{y}}$. $t(1+t^{-1}s)t^{-1} = (t+s)t^{-1} = 1+st^{-1},$ so that $(t(1+t^{-1}s)^{-1}t^{-1})^{-1} \equiv 1$ (mod $J$); thus $t(1+t^{-1}s)^{-1}t^{-1} \equiv 1$ (mod $J$), so that ${y}' \equiv t(1+t^{-1}s)^{-1}t^{-1}{y}' = {y}$ (mod $J$).
\end{proof}
We now relate the above commutative diagram to the estimate we must make:
\begin{lemma}
If ${x}\in R^d$ is chosen so that the sum defining $A_t\chi(x)$ is nonempty, then we have that:

$$A_{t}\chi({x}) = |\ker(p^d)\cap\ker(\varphi)| A_{\tilde{t}}\tilde{\chi}({\tilde{x}}).$$
\end{lemma}
\begin{proof}
Recall that 
$$A_t\chi(x) = \sum\limits_{{\tilde{y}}\cdot {\tilde{x}} = \tilde{t}}\tilde{\chi}({\tilde{y}})\sum\limits_{\substack{{y}: y \equiv {\tilde{y}} \text{ (mod }J)\\ \ {y}\cdot {x} = t}}1.$$

We must show that when $x$ is chosen so that the sum defining $A_t$ is not empty,
$$\sum\limits_{\substack{{y}: {y} \equiv \tilde{y} \text{ (mod }J) \\ {y}\cdot {x} = t}}1 = |\ker(p^d)\cap\ker(\varphi)|.$$
The sum defining $A_t$ is nonempty if and only if $\varphi$ is surjective, if and only if $\tilde{\varphi}$ is surjective. For such values of $x$, we have by lemma \ref{congruenceConditions} that for each ${\tilde{y}}$ with ${\tilde{y}}\cdot{\tilde{x}} = \tilde{t},$ the sum

$$\sum\limits_{\substack{{y}: {y} \equiv {\tilde{y}} \text{ (mod }J) \\ {y}\cdot {x} = t}}1$$

is nonempty. Further, given ${y}$ and ${y}'$ both congruent to $\tilde{y}$ (mod $J^d$) and such that ${y}\cdot {x} = t$ and ${y}' \cdot x = t$, we have that ${y}-{y}'\in J^d$ and $\varphi({y} - {y}') = 0$, and given any ${r}\in \ker(p^d)\cap\ker(\varphi)$, ${y} + {r}$ is congruent to $\tilde{y}$ (mod $J^d$) and $({y}+{r})\cdot x = t$, so that 

$$\sum\limits_{\substack{{y}: {y} \equiv \tilde{y} \text{ (mod }J) \\ {y}\cdot {x} = t}}1 = |\ker(p^d)\cap\ker(\varphi)|.$$
\end{proof}
It follows that now we must only compute $|\ker(p^d)\cap\ker(\varphi)|.$ In fact, since we only need a lower bound on $||A_t||_2$, we will only need the following lower bound on $|\ker(p^d)\cap\ker(\varphi)|$:
\begin{lemma}
Suppose that $x$ is chosen so that $\varphi$ (equivalently, $\tilde{\varphi}$) is surjective. Then 

$$|\ker(p^d)\cap\ker(\varphi)| \geq |J|^{d-1}.$$
\end{lemma} 
\begin{proof}
Pick $r_j$ for $j = 1, 2,\dots, |R|$ so that $r_j + \ker(\varphi)$ are the distinct (additive) cosets of $\ker(\varphi)$ in $R^d$. Then 
$$R^d = \bigcup\limits_j\left(r_j + \ker(\varphi)\right),$$ 
so that 
$$\ker(p^d) = \bigcup_{j} \left(\ker(p^d)\cap \left(r_j + \ker(\varphi)\right)\right).$$

For each $r_j$, if $\ker(p^d)\cap \left(r_j + \ker(\varphi)\right) \neq \emptyset,$ we fix 
$$s_j\in \ker(p^d)\cap (r_j + \ker(\varphi))$$
and write $r_j + \ker(\varphi) = s_j + \ker(\varphi),$ so that 
$$\ker(p^d)\cap (r_j + \ker(\varphi)) = \ker(p^d)\cap (s_j + \ker(\varphi)) = s_j + \ker(p^d)\cap \ker(\varphi).$$
In particular, for all $j$ such that $\ker(p^d)\cap (r_j + \ker(\varphi)) \neq \emptyset$, we have that 
$$|\ker(p^d)\cap (r_j + \ker(\varphi))| = |\ker(p^d)\cap \ker(\varphi)|.$$
Let $N_x$ denote the number of $r_j$ such that $\ker(p^d)\cap\left( r_j + \ker(\varphi)\right) \neq \emptyset.$ It follows that 
$$|\ker(p^d)\cap\ker(\varphi)| = \frac{|J|^d}{N_{{x}}}.$$
We estimate $N_x$ as follows: for a fixed $r_j$, we have 
$$p\circ\varphi(r_j + \ker(\varphi)) = p(\varphi(r_j) + {0}) = p(\varphi(r_j)).$$
Now, $p\circ\varphi = \tilde\varphi\circ p^d$, so for any 
$${y}\in r_j + \ker(\varphi),$$
we have $\tilde\varphi\circ p^d({y}) = p\circ\varphi(r_j).$ If $\ker(p^d)\cap \left(r_j + \ker(\varphi)\right)$ is nonempty, we can pick 
$$y \in \ker(p^d)\cap \left(r_j + \ker(\varphi)\right).$$
Since $p^d(y) = 0$, we have that $p\circ \varphi(r_j) = 0.$ Thus, unless $p\circ\varphi(r_j) = 0,$ we have that $\ker(p^d)\cap\left( r_j + \ker(\varphi)\right) = \emptyset,$ so that clearly
$$N_x\leq \#\{r_j|p\circ\varphi(r_j) = 0\}.$$ 
If $p\circ\varphi(r_j) = 0,$ $\varphi(r_j)\in J$, so that by the first isomorphism theorem, exactly $|J|$ of the $r_j$ have the property that $p\circ \varphi(r_j) = 0.$  It follows that:

$$N_x \leq |J|,$$

from which we have
$$|\ker(p^d)\cap\ker(\varphi)| \geq |J|^{d-1}.$$
\end{proof}
We may now finally compute $||(A_t\chi)||_2,$ using $\varphi_{{x}}$ to keep track of the dependence of the map $\varphi$ on the choice of ${x},$ while keeping in mind that each ${\tilde{x}}\in (R/J)^d$ has precisely $|J|^d$ ${x}$'s in $R^d$ that reduce to it modulo $J$: 
\[\begin{split}
    ||(A_t\chi)||_2 &= \left(\sum\limits_{{x}\in R^d}\left\lvert(A_t\chi)({x})\right\rvert^2\right)^{1/2}\\
    &=\left(\Slim{\tilde{x}\in(R/J)^d}\Slim{\substack{x\in R^d\\x\equiv \tilde{x}\text{ (mod } J)}}\left[|\ker(p^d)\cap\ker(\varphi_{{x}})|\Abs{(A_{\tilde t}\tilde\chi)({\tilde{x}})}\right]^2\right)^{1/2}\\
    &\geq \left(\Slim{\tilde{x}\in(R/J)^d}\Slim{\substack{x\in R^d\\x\equiv \tilde{x}\text{ (mod } J)}}\left[|J|^{d-1}\Abs{(A_{\tilde{t}}\tilde\chi)({\tilde{x}})}\right]^2\right)^{1/2}\\
    &=\left(|J|^d\sum\limits_{{\tilde x}\in (R/J)^d}\left[|J|^{d-1}\Abs{(A_{\tilde{t}}\tilde\chi)({\tilde{x}})}\right]^2\right)^{1/2}\\
    &= |J|^{d/2}|J|^{d-1}||A_{\tilde{t}}\tilde\chi||_2.
\end{split}\]
In particular, we have
$$\frac{||(A_t\chi)||_2}{||\chi||_2} = \frac{||(A_t\chi)||_2}{|R|^{d/2}} \geq \frac{|J|^{d/2}|J|^{d-1}||A_{\tilde{t}}\tilde\chi||_2}{|J|^{d/2}|R/J|^{d/2}} = |J|^{d-1}\frac{||(A_{\tilde{t}}\tilde\chi)||_2}{||\tilde\chi||_2}.$$
This completes the proof of theorem \ref{JacobsonBound}.
\end{proof}
We are finally in position to prove our theorem for general finite rings.
\begin{theorem}\label{GeneralCaseDots}
    Fix a constant $C>0$, and an integer $d\geq 2$. There are only finitely many pairs $(R, t)$ with $R$ a finite ring and $t\in R^\times$ where $R$ is not a finite field and the $d$-dimensional $t$-incidence-Salem number of $R$ is at most $C$.
\end{theorem}
\begin{proof}
    Let $R$ be a finite ring, $t\in R^\times$, and assume that the $t$-incidence-Salem number of $R$ is at most $C$. Let $J$ be the Jacobson radical of $R$. If $|J| = 1,$ then $R$ is semisimple, and this is just theorem \ref{semisimple}. Suppose then that $J$ is nontrivial. Let $\tilde \chi$ be a nontrivial character of $(R/J)^d,$ and let $\chi$ be its pullback to $R^d.$ By theorem \ref{JacobsonBound}, we have:

$$\frac{||(A_t\chi)||_2}{|R|^{\frac{d-1}{2}}||\chi||_2} \geq \frac{|J|^{d-1}||(A_{\tilde{t}}\tilde\chi)||_2}{|J|^\frac{d-1}{2}|R/J|^\frac{d-1}{2}||\tilde\chi||_2} = |J|^\frac{d-1}{2}\frac{||(A_{\tilde{t}}\tilde\chi)||_2}{|R/J|^\frac{d-1}{2}||\tilde\chi||_2}.$$

    In particular, since the $t$-incidence-Salem number of $R$ is at most $C$, we have that:

    $$|J|^\frac{d-1}{2}\frac{||(A_{\tilde{t}}\tilde\chi)||_2}{|R/J|^\frac{d-1}{2}||\tilde\chi||_2}\leq C.$$

    First, choosing $\tilde\chi$ to be the character from corollary \ref{lowerBoundSemisimple} we have that
    $$\frac{||(A_{\tilde{t}}\tilde\chi)||_2}{|R/J|^\frac{d-1}{2}||\tilde\chi||_2}\geq\frac{1}{2},$$
    so that $|J|\leq (2C)^\frac{2}{d-1}$ is bounded. 
    
    As $|J|\geq 1,$ we also have that for all nontrivial characters $\tilde\chi$ of $(R/J)^d,$

    $$\frac{||(A_{\tilde{t}}\tilde\chi)||_2}{|R/J|^\frac{d-1}{2}||\tilde\chi||_2}\leq C.$$

    As $R/J$ is semisimple, by the proof of theorem \ref{semisimple}, there are only finitely many possibilities for $R/J$ other than a field. Coupled with the bound on $|J|$, we have shown that there are only finitely many finite rings $R$ with $t$-incidence-Salem number at most $C$ and semisimple quotient $R/J$ not a field.

    Now, suppose $R/J = \mathbb{F}_q.$ We shall show that for $q$ large enough, there can be no finite ring $R$ with nontrivial Jacobson radical $J$ such that $|J|\leq (2C)^\frac{2}{d-1}$. This follows from a quick application of Nakayama's lemma (\cite{Rotman}, lemma C-2.8): unless $J = 0$, $J/J^2\neq 0,$ and hence $J/J^2$ is a nontrivial $R/J = \mathbb{F}_q$ vector space. Then $q$ divides $|J|;$ for $q>(2C)^\frac{2}{d-1}$ this will be impossible. Thus, for rings with nontrivial Jacobson radical, $q$ is bounded; bounded $|J|$ and bounded $q$ means bounded $|R| = q|J|$, and only finitely many possibilities for $R$ and $t.$
\end{proof}

\section{Acknowledgments}
The author is grateful to Professors Alex Iosevich and Azita Mayeli for introducing him to and contextualizing this problem, and for their support in writing this paper. The author completed this work while partially supported by a Graduate Center Fellowship at the CUNY Graduate Center.

\end{document}